\documentclass{amsart}
\usepackage{amssymb}
\usepackage{float}
\usepackage{array}
\usepackage[margin=1in]{geometry}
\makeatletter
\newcommand{\thickhline}{%
    \noalign {\ifnum 0=`}\fi \hrule height 1pt
    \futurelet \reserved@a \@xhline
}
\newcolumntype{"}{@{\hskip\tabcolsep\vrule width 1pt\hskip\tabcolsep}}
\makeatother

\newtheorem{theorem}{Theorem}[section]

\theoremstyle{definition}

\newtheorem{corollary}[theorem]{Corollary}

\newtheorem{example}[theorem]{Example}

\theoremstyle{remark}
\newtheorem{remark}[theorem]{Remark}

\numberwithin{equation}{section}

\def\m{{\mathfrak m}}

\begin{document}
\title{4-generated pseudo symmetric monomial curves with not Cohen-Macaulay tangent cones}
\author{N\.{i}l \c{S}ah\.{i}n}
\address{Department of Industrial Engineering, Bilkent University, Ankara, 06800 Turkey}
\email{nilsahin@bilkent.edu.tr}

\thanks{}


\date{\today}

\commby{}

\dedicatory{}

\begin{abstract}In this article, standard bases of some toric ideals associated to 4-generated pseudo symmetric semigroups with not Cohen-Macaulay tangent cones at the origin are computed. As the tangent cones are not Cohen-Macaulay, non-decreasingness of the Hilbert function of the local ring was not guaranteed. Therefore, using these standard bases, Hilbert functions are explicitly computed as a step towards the characterization of Hilbert function. In addition, when the smallest integer satisfying $k(\alpha_2+1)<(k-1)\alpha_1+(k+1)\alpha_{21}+\alpha_3$ is $1$, it is proved that the Hilbert function of the local ring is non-decreasing. 

\end{abstract}

\maketitle

\section{introduction}\label{1}
Let $n_1< n_2<\dots<n_k$ be positive integers with $\gcd (n_1,\dots,n_k)=1$ and let $S$ be the numerical semigroup $S=\langle n_1,\dots,n_k \rangle=\{ \displaystyle\sum_{i=1}^{k} u_in_i | u_i \in \mathbb{N}\}$. $K$ being an algebraically closed field, let $K[S]=K[t^{n_1}, t^{n_2}, \dots, t^{n_k}]$ be the semigroup ring of $S$ and $A=K[X_1,X_2,\dots,X_k]$. If  $\phi: A {\longrightarrow} K[S]$ with $\phi(X_i)=t^{n_i}$ and $\ker \phi=I_S$ , then $K[S]\simeq A/I_S$. If we denote the affine curve with parametrization 
$$X_1=t^{n_1},\ \ X_2=t^{n_2},\ \dots,\  X_k=t^{n_k}  $$
corresponding to $S$ with $C_S$, then $I_S$ is called the defining ideal of $C_S$. The smallest integer $n_1$ in the semigroup is called the \textit{multiplicity} of $C_S$. Denote the corresponding local ring with $R_S=K[[t^{n_1},\dots,t^{n_k}]]$ and say the maximal ideal is $\m=\langle t^{n_1},\dots,t^{n_k}\rangle$. Then
$gr_{\mathfrak{m}}(R_S)=\bigoplus_{i=0}^{\infty} \mathfrak{m}^i/\mathfrak{m}^{i+1}\cong A/{I^*_S},$ is the associated graded ring 
where ${I^*_S}=\langle f^*| f \in I_S \rangle$ with $f^*$ denoting the least homogeneous summand of $f$. 

Hilbert function $H_{R_S}(n)$ of the local ring $R_S$ with the maximal ideal
$\mathfrak{m}$ is defined to be the Hilbert function of the associated
graded ring $gr_{\mathfrak{m}}(R_S)=\bigoplus_{i=0}^{\infty} \mathfrak{m}^i/\mathfrak{m}^{i+1}$.
In other words,
$$H_{R_S}(n)=H_{gr_{\mathfrak{m}}(R_S)}(n)=dim_{R_S/\mathfrak{m}}(\mathfrak{m}^n/\mathfrak{m}^{n+1}) \; \; n\geq
0.$$
This function is called non-decreasing if $H_{R_S}(n)\geq H_{R_S}(n-1)$ for all $n \in \mathbb{N}$. If $\exists l \in \mathbb{N}$ such that $H_{R_S}(l)> H_{R_S}(l-1)$ then it is called decreasing at level $l$.
The Hilbert series of $R_S$ is defined to be the generating function
$$HS_{R_S}(t)=\begin{displaystyle}\sum_{n \in \mathbb{N}}\end{displaystyle}H_{R_S}(n)t^n.$$
By the Hilbert-Serre theorem it can also be written as:
$HS_{R_S}(t)=\frac{P(t)}{(1-t)^k}=\frac{Q(t)}{(1-t)^d}$, where $P(t)$ and $Q(t)$ are polynomials with coefficients from
$\mathbb{Z}$ and $d$ is the Krull dimension of $R_S$. $P(t)$ is called first Hilbert Series and $Q(t)$ is called second Hilbert series, \cite{greuel-pfister,Rossi}. It is also known that there is
a polynomial $P_{R_S}(n) \in \mathbb{Q}[n]$ called the Hilbert polynomial of $R_S$ such that
$H_{R_S}(n)=P_{R_S}(n)$ for all $n \geq n_0$, for some $n_0 \in \mathbb{N}$. The smallest $n_0$ satisfying this
condition is the regularity index of the Hilbert function of $R_S$. A natural question is whether the Hilbert function of the local ring is non-decreasing. In general Cohen-Macaulayness of a one dimensional local ring does not guarantee the nondecreasingness of its Hilbert Function for embedding dimensions greater than three. However, it is known that if the tangent cone is Cohen-Macaulay, then the Hilbert function of the local ring is non-decreasing.  When the tangent cone is not Cohen-Macaulay, Hilbert function of the local ring may decrease at some level.
Rossi's conjecture states that  ``Hilbert function of a Gorenstein Local ring of dimension one is non-decreasing''. This conjecture is still open in embedding dimension 4 even for the monomial curves. It is known that the local ring corresponding to monomial curve is Gorenstein iff corresponding numerical semigroup is symmetric. Arslan and Mete in \cite{am} proved the conjecture is true for monomial curves in affine $4$-space with symmetric semigroups under the numerical condition  ``$\alpha_2\leq\alpha_{21}+\alpha_{24}$" on the generators of the semigroup by proving that the tangent cone is Cohen-Macaulay. The conjecture is open for local rings corresponding to 4-generated symmetric semigroups with $\alpha_{2}>\alpha_{21}+\alpha_{24}$. For some recent work on the monotonicity of the Hilbert Function see \cite{am,AMS,AOS,mz,Oneto,Puthen}.

Since symmetric and pseudo-symmetric semigroups are maximal with respect to inclusion with fixed genus, a natural question is whether Rossi's conjecture is even true for local rings corresponding to 4-generated pseudo-symmetric semigroups. The main aim of this paper is to give a characterization to the Hilbert functions of the local rings of some 4 generated pseudo- symmetric monomial curves. In \cite{SahinSahin}, we showed that if $\alpha_2 \leq \alpha_{21}+1$, then the tangent cone is Cohen-Macaulay, and hence, the Hilbert function of the local ring is non-decreasing. In this paper, we focus on the open case of 4-generated pseudo symmetric monomial curves with $\alpha_{2}>\alpha_{21}+1$. For the computation of the first Hilbert series of the tangent cone, a standard basis based computation with the algorithm in \cite{bayer} will be used. 

Recall from \cite{komeda} that a $4$-generated semigroup $S=\langle n_1,n_2,n_3,n_4 \rangle$ is pseudo-symmetric if and only if there are integers $\alpha_i>1$, for
$1\le i\le4$, and $\alpha_{21}>0$ with $\alpha_{21}<\alpha_1-1$,
such that 
 \begin{eqnarray*}
n_1&=&\alpha_2\alpha_3(\alpha_4-1)+1,\\
n_2&=&\alpha_{21}\alpha_3\alpha_4+(\alpha_1-\alpha_{21}-1)(\alpha_3-1)+\alpha_3,\\
n_3&=&\alpha_1\alpha_4+(\alpha_1-\alpha_{21}-1)(\alpha_2-1)(\alpha_4-1)-\alpha_4+1,\\
n_4&=&\alpha_1\alpha_2(\alpha_3-1)+\alpha_{21}(\alpha_2-1)+\alpha_2. 
\end{eqnarray*}
Then, the toric ideal is $I_S=\langle f_1,f_2,f_3,f_4,f_5 \rangle$ with
\begin{eqnarray*} f_1&=&X_1^{\alpha_1}-X_3X_4^{\alpha_4-1},  \quad \quad
f_2=X_2^{\alpha_2}-X_1^{\alpha_{21}}X_4, \quad
f_3=X_3^{\alpha_3}-X_1^{\alpha_1-\alpha_{21}-1}X_2,\\
f_4&=&X_4^{\alpha_4}-X_1X_2^{\alpha_2-1}X_3^{\alpha_3-1}, \quad 
f_5=X_1^{\alpha_{21}+1}X_3^{\alpha_3-1}-X_2X_4^{\alpha_4-1}.
\end{eqnarray*}  
If $n_1<n_2<n_3< n_4$ then it is known from \cite{SahinSahin} that 
\begin{enumerate}
\item $\alpha_1>\alpha_4$
\item $\alpha_3<\alpha_1-\alpha_{21}$
\item $\alpha_4<\alpha_2+\alpha_3-1$
\end{enumerate}
and these conditions completely determine the leading monomials of $f_1, f_3$ and $f_4$. Indeed, $ {\rm LM}( f_1)= X_3X_4^{\alpha_{4}-1}$ by $(1)$, ${\rm LM}(f_3)= X_3^{\alpha_3}$ by $(2)$, ${\rm LM}(f_4)=X_4^{\alpha_4}$ by $(3)$
If we also let 
\begin{enumerate}
\item[(4)]  $\alpha_2>\alpha_{21}+1$
\end{enumerate}
then ${\rm LM}(f_2)=X_1^{\alpha_{21}}X_4$ by $(4)$. To determine the leading monomial of $f_5$ we need the following remark.

\begin{remark}
Let $n_1<n_2<n_3<n_4$. Then $(4)$ implies
\begin{enumerate}
\item[(5)] $\alpha_{21}+\alpha_3 > \alpha_4$ 
\end{enumerate}

\end{remark}
\begin{proof}
Assume to the contrary $\alpha_{21}+\alpha_3 \leq \alpha_4$. Then, $\alpha_4=\alpha_{21}+\alpha_3+n,\ \ (*),$  for some nonnegative $n$. Then from  $(1)$, $\alpha_1=\alpha_{21}+\alpha_3+n+m\ \ (**),$ for some positive $m$ and hence
$\alpha_2=\alpha_{21}+1+k\ \ (***),$ for some positive $k$ by $(4)$. Then from $n_1 < n_2$ we have:\\
$\alpha_2\alpha_3(\alpha_4-1)+1< \alpha_{21}\alpha_3\alpha_4+(\alpha_1-\alpha_{21}-1)(\alpha_{3}-1)+\alpha_3$. Using $(***)$, we have:\\
$\alpha_{21}\alpha_3\alpha_4+(1+k)\alpha_3\alpha_4-\alpha_3(\alpha_{21}+1+k)+1 < \alpha_{21}\alpha_3\alpha_4+(\alpha_1-\alpha_{21}-1)(\alpha_3-1)\alpha_3$,\\
$\alpha_3((1+k)\alpha_4-\alpha_{21}-1-k-\alpha_1+\alpha_{21}+1-1)+1<1+\alpha_{21}-\alpha_1$. Then from $(**)$:\\
$\alpha_3(\alpha_4+k\alpha_4-1-k-\alpha_1)<-\alpha_3-n-m$. Using $(*)$ and $(**)$ again, we get:\\
$\alpha_3(m+k\alpha_4-k)<-n-m<0$. As $\alpha_3>0$, we have $m+k(\alpha_4-1)<0$ which is a contradiction as each term in the sum is positive.
\end{proof}

\begin{remark}
Let $n_1<n_2<n_3<n_4$. Then $(1)$ and $(4)$ implies
\begin{enumerate}
\item[(6)] $\alpha_1+\alpha_{21}+1\geq \alpha_2+\alpha_4$
\end{enumerate} 
  
\end{remark}
\begin{proof}
Assume to the contrary $\alpha_1+\alpha_{21}+1< \alpha_2+\alpha_4$. We know from $(1)$ and $(4)$ that $\alpha_1=\alpha_4+m$ and $\alpha_2=\alpha_{21}+1+k$, for some positive $m$ and $k$. Hence
\begin{eqnarray*}
\alpha_1+\alpha_{21}+1<\alpha_2+\alpha_4&\Rightarrow& \alpha_4+m+\alpha_{21}+1<\alpha_{21}+1+k+\alpha_4\\
&\Rightarrow& m<k
\end{eqnarray*}
On the other hand, from $n_1<n_2$ we have:\\
$\alpha_2\alpha_3(\alpha_4-1)+1< \alpha_{21}\alpha_3\alpha_4+(\alpha_1-\alpha_{21}-1)(\alpha_{3}-1)+\alpha_3$. From $(1)$ and $(4)$:\\
$(\alpha_{21}+1+k)\alpha_3(\alpha_4-1)+1< \alpha_{21}\alpha_3\alpha_4+(\alpha_4-\alpha_{21}+m-1)(\alpha_{3}-1)+\alpha_3$. Expanding this, we obtain\\
$k\alpha_3\alpha_4+\alpha_4+m<\alpha_3(m+1+k)+\alpha_{21}$. Now as $\alpha_4+m=\alpha_1$ and $\alpha_{21}<\alpha_1$:\\
$k\alpha_3\alpha_4+\alpha_1<\alpha_3(m+1+k)+\alpha_{21}<\alpha_3(m+1+k)+\alpha_{1}\Rightarrow k\alpha_3\alpha_4<\alpha_3(m+k+1) \Rightarrow k\alpha_4<m+k+1$\\
$\Rightarrow k(\alpha_4-1)<m+1<k+1$ which is a contradiction as $\alpha_4-1\geq 1$.Hence, $\alpha_1+\alpha_{21}+1\geq \alpha_2+\alpha_4.$

\end{proof}
Being able to determine the leading monomials of the generators of $I_S$ in the open case, lots of different possibilities must be considered for the leading monomials of the s-polynomials appearing in standard bases computation. Unfortunately, there is not a general form for the standard basis if $\alpha_2>\alpha_{21}+1$ contrary to the case $\alpha_2\leq\alpha_{21}+1$. Though there are five elements in minimal standard basis of $I_S$ if $\alpha_2\leq\alpha_{21}+1$ (see \cite{SahinSahin}), the number of elements in the standard basis increases as $\alpha_4$ increases if $\alpha_2>\alpha_{21}+1$.
\begin{example} The following examples are done using SINGULAR, see \cite{singular}.
\begin{itemize} 
\item Standard basis for $\alpha_{21}=8 , \alpha_1=16 , \alpha_2=20 ,\alpha_3=7, \alpha_{4}=2 $ is $\{ X_1^{16}-X_3X_4,
X_2^{20}-X_1^8X_4,
X_3^7-X_1^7X_2,
X_4^2-X_1X_2^{19}X_3^6,
X_1^9X_3^6-X_2X_4,
X_1^{24}-X_2^{20}X_3,
X_1^{17}X_3^6-X_2^{21} \}$
\item Standard basis for $\alpha_{21}=2 , \alpha_1=9 , \alpha_2=5 ,\alpha_3=3, \alpha_{4}=3 $ is $\{X_3^3-X_1^6X_2, X_1^2X_4-X_2^5, X_2X_4^2-X_1^3X_3^2, X_3X_4^2-X_1^9, X_4^3-X_1X_2^4X_3^2, X_1^5X_3^2-X_2^6X_4, X_2^5X_3X_4-X_1^{11}, X_2^{10}X_3-X_1^{13},X_2^{16}X_4-X_1^{18}X_3, X_2^{21}-X_1^{20}X_3 \}$

\item Standard basis for $\alpha_{21}=8 , \alpha_1=16 , \alpha_2=11 ,\alpha_3=3, \alpha_{4}=5 $ is $\{X_3^3-X_1^7X_2,X_2X_4^4-X_1^9X_3^2, X_3X_4^4-X_1^{16},X_4^5-X_1X_2^{10}X_3^2, X_1^8X_4-X_2^{11}, X_2^{12}X_4^3-X_1^{17}X_3^2, X_2^{11}X_3X_4^3-X_1^{24}, X_2^{23}X_4^2-X_1^{25}X_3^2, X_2^{22}X_3X_4^2-X_1^{32}, X_1^{33}X_3^2-X_2^{34}X_4, X_2^{33}X_3X_4-X_1^{40},  X_2^{44}X_3-X_1^{48}, X_2^{78}X_4-X_1^{81}X_3, X_2^{89}-X_1^{89}X_3\}$
\end{itemize}

\end{example}
Hence, we focus on the case when $\alpha_4=2$
\section{Standard bases}\label{2}

\begin{theorem}
Let $S=\langle n_1,n_2,n_3,n_4 \rangle$ be a 4-generated pseudosymmetric numerical semigroup with $n_1<n_2<n_3<n_4$ and $\alpha_2>\alpha_{21}+1$. If $\alpha_4=2$ and $k$ is the smallest positive integer such that $k(\alpha_2+1)<(k-1)\alpha_1+(k+1)\alpha_{21}+\alpha_3$ then the standard basis for $I_S$ is
$$\{f_1,f_2,f_3,f_4,f_5,f_6,...,f_{6+k}\}$$
where $f_6=X_1^{\alpha_1+\alpha_{21}}-X_2^{\alpha_2}X_3$ and $f_{6+j}=X_1^{(j-1)\alpha_1+(j+1)\alpha_{21}+1}X_3^{\alpha_3-j}-X_2^{j\alpha_2+1}$ for $j=1,2,...,k$
\end{theorem}
\begin{proof}
We will prove the theorem using induction on $k$ and apply standard basis algorithm with NFM\tiny{ORA} \normalsize  as the normal form algorithm, see \cite{greuel-pfister}. Here $T_h$ denotes the set $\{g \in G : {\rm LM}(g) \mid {\rm LM}(h)\}$ and ${\rm ecart}(h)$ is
${deg}(h)-{deg}({\rm LM}(h))$. Note that ${\rm LM}(f_6)=X_2^{\alpha_2}X_3$ by $(6)$.

\begin{center}
\underline{ For $k=1$:}
\end{center}
In this case $f_7=X_1^{2\alpha_{21}+1}X_3^{\alpha_3-1}-X_2^{\alpha_2+1}$ and $\alpha_2+1<2\alpha_{21}+\alpha_3$ which implies that ${\rm LM}(f_7)=X_2^{\alpha_2+1}$ . We need to show that $NF(\text{spoly}(f_m,f_n) \vert G)=0$ for all $m,n$ with $1 \leq m < n \leq 7$.

\begin{itemize}
\item  ${\rm spoly}(f_1,f_2)=f_6$ and hence  $NF({\rm spoly}(f_1,f_2) \vert G)=0$
\item   ${\rm spoly}(f_1,f_3)=X_1^{\alpha_{1}}X_3^{\alpha_3-1}-X_1^{\alpha_1-\alpha_{21}-1}X_2X_4$ and ${\rm LM}({\rm spoly}(f_1,f_{3}))=X_1^{\alpha_1-\alpha_{21}-1}X_2X_4$ by $(5)$. Let $h_1={\rm spoly}(f_1,f_3)$. If $\alpha_{1}<2\alpha_{21}+1$ or $2\alpha_{21}+\alpha_{3}\leq \alpha_{2}+1$ then $T_{h_1}=\{f_5\}$ and since ${\rm spoly}(h_1,g)=0$,  $NF({\rm spoly}(f_1,f_3) \vert G)=0$. Otherwise $T_{h_1}=\{f_2\}$ and ${\rm spoly}(h_1,f_2)=X_1^{\alpha_1-2\alpha_{21}-1}X_2^{\alpha_2+1}-X_1^{\alpha_{1}}X_3^{\alpha_{3}-1}$. Set $h_2={\rm spoly}(h_1,f_2)$, ${\rm LM}(h_2)=X_1^{\alpha_1-2\alpha_{21}-1}X_2^{\alpha_2+1}$ and $T_{h_2}=\{f_7\}$ and  ${\rm spoly}(h_2,f_7)=0$, hence $NF({\rm spoly}(f_1,f_3) \vert G)=0$ . 
\item   ${\rm spoly}(f_1,f_4)=X_1^{\alpha_1}X_4-X_1X_2^{\alpha_{2}-1}X_3^{\alpha_3}$. Set $h_1={\rm spoly}(f_1,f_4)$. If ${\rm LM}(h_1)=X_1^{\alpha_1}X_4$ then $T_{h_1}=\{f_2\}$ and ${\rm spoly}(h_1,f_2)=X_1X_2^{\alpha_2-1}f_3$. If ${\rm LM}(h_1)=X_1X_2^{\alpha_{2}-1}X_3^{\alpha_3}$ then $T_{h_1}=\{f_3\}$ and ${\rm spoly}(h_1,f_3)=X_1^{\alpha_1-\alpha_{21}}f_2$. Hence in both cases, $NF({\rm spoly}(f_1,f_4) \vert G)=0$
\item   ${\rm spoly}(f_1,f_5)=X_1^{\alpha_{21}+1}X_3^{\alpha_3}-X_1^{\alpha_{1}}X_2=X_1^{\alpha_{21}+1}f_3$ hence $NF({\rm spoly}(f_1,f_5) \vert G)=0$
\item   ${\rm spoly}(f_1,f_6)=X_1^{\alpha_1}f_2$ and hence $NF({\rm spoly}(f_1,f_6) \vert G)=0$
\item  ${\rm spoly}(f_1,f_7)=X_1^{\alpha_1+2\alpha_{21}+1}X_3^{\alpha_3-2}-X_2^{\alpha_2+1}X_4$ if $\alpha_3+2\alpha_{21}<\alpha_2+1$. Set $h_1={\rm spoly}(f_1,f_7)$. Using $(6)$ and the fact that $\alpha_{21}+\alpha_3>2$, we can conclude that ${\rm LM}(h_1)=X_2^{\alpha_2+1}X_4$ and $T_{h_1}=\{f_5\}$ then ${\rm spoly}(h_1,f_5)=X_1^{\alpha_{21}+1}X_3^{\alpha_3-2}f_6$ and hence $NF({\rm spoly}(f_1,f_7) \vert G)=0$. \newline \noindent
$NF({\rm spoly}(f_1,f_7) \vert G)=0$ if $\alpha_3+2\alpha_{21}>\alpha_2+1$, as ${\rm LM}(f_1)$ and ${\rm LM}(f_7)$ are relatively prime.
\item $NF({\rm spoly}(f_2,f_3) \vert G)=0$ as ${\rm LM}(f_2)$ and ${\rm LM}(f_3)$ are relatively prime.
\item   ${\rm spoly}(f_2,f_4)=X_2^{\alpha_2-1}f_5$ and hence $NF({\rm spoly}(f_2,f_4) \vert G)=0$
\item   ${\rm spoly}(f_2,f_5)=f_7$ and hence $NF({\rm spoly}(f_2,f_5) \vert G)=0$
\item  $NF({\rm spoly}(f_2,f_6) \vert G)=0$ as ${\rm LM}(f_2)$ and ${\rm LM}(f_6)$ are relatively prime. 
\item  ${\rm spoly}(f_2,f_7)=X_2^{\alpha_2}f_5$ if $2\alpha_{21}+\alpha_{3}<\alpha_2+1$. Otherwise ${\rm LM}(f_2)$ and ${\rm LM}(f_7)$ are relatively prime. Hence, in both cases $NF({\rm spoly}(f_2,f_7) \vert G)=0$.
\item $NF({\rm spoly}(f_3,f_4) \vert G)=0$ as ${\rm LM}(f_3)$ and ${\rm LM}(f_4)$ are relatively prime.
\item  $NF({\rm spoly}(f_3,f_5) \vert G)=0$ as ${\rm LM}(f_3)$ and ${\rm LM}(f_4)$ are relatively prime.
\item   ${\rm spoly}(f_3,f_6)=X_1^{\alpha_1-\alpha_{21}-1}f_6$ and hence $NF({\rm spoly}(f_3,f_6) \vert G)=0$
\item   ${\rm spoly}(f_3,f_7)=X_2f_6$ if $2\alpha_{21}+\alpha_{3}<\alpha_2+1$. Otherwise ${\rm LM}(f_3)$ and ${\rm LM}(f_7)$ are relatively prime. Hence, in both cases $NF({\rm spoly}(f_3,f_7) \vert G)=0$
\item   ${\rm spoly}(f_4,f_5)=X_1X_3^{\alpha_3-1}f_2$ and hence $NF({\rm spoly}(f_4,f_5) \vert G)=0$
\item $NF({\rm spoly}(f_4,f_6) \vert G)=0$ as ${\rm LM}(f_4)$ and ${\rm LM}(f_6)$ are relatively prime.
\item $NF({\rm spoly}(f_4,f_7) \vert G)=0$ as ${\rm LM}(f_4)$ and ${\rm LM}(f_7)$ are relatively prime.
\item   ${\rm spoly}(f_5,f_6)=X_1^{\alpha_{21}+1}X_2^{\alpha_2-1}X_3^{\alpha_3}-X_1^{\alpha_{1}+\alpha_{21}}X_4$ and let $h_1= {\rm spoly}(f_5,f_6)$. If ${\rm LM}(h_1)=X_1^{\alpha_{21}+1}X_2^{\alpha_2-1}X_3^{\alpha_3}$ then $T_{h_1}=\{f_3\}$ and ${\rm spoly}(h_1,f_3)=X_1^{\alpha_1-\alpha_{21}-1}f_2$ and hence $NF({\rm spoly}(f_5,f_6) \vert G)=0$. If ${\rm LM}(h_1)= X_1^{\alpha_{1}+\alpha_{21}}X_4$ then $T_{h_1}=\{f_2\}$ and ${\rm spoly}(h_1,f_2)=X_2^{\alpha_2-1}f_3$ and hence $NF({\rm spoly}(f_5,f_6) \vert G)=0$
\item  $NF({\rm spoly}(f_5,f_7) \vert G)=0$ if $\alpha_3+2\alpha_{21}<\alpha_2+1$, as ${\rm LM}(f_5)$ and ${\rm LM}(f_7)$ are relatively prime.  Otherwise ${\rm spoly}(f_5,f_7)=X_1^{\alpha_{21}+1}X_3^{\alpha_{3}-1}f_2$ and hence $NF({\rm spoly}(f_5,f_7) \vert G)=0$
\item  ${\rm spoly}(f_6,f_7)=f_8$ if $\alpha_3+2\alpha_{21}<\alpha_2+1$ and hence $NF({\rm spoly}(f_6,f_7) \vert G)=0$. Otherwise ${\rm spoly}(f_6,f_7)=X_1^{2\alpha_{21}+1}f_3$ and hence $NF({\rm spoly}(f_6,f_7) \vert G)=0$
\end{itemize}
\begin{center}
\underline{Assume  the statement is true for $k<l$:}
\end{center} If $k$ is the smallest positive integer such that $k(\alpha_2+1)<(k-1)\alpha_1+(k+1)\alpha_{21}+\alpha_3$ then the standard basis for $I_S$ is
$$\{f_1,f_2,f_3,f_4,f_5,f_6,...,f_{6+k}\}$$
where $f_6=X_1^{\alpha_1+\alpha_{21}}-X_2^{\alpha_2}X_3$ and $f_{6+j}=X_1^{(j-1)\alpha_1+(j+1)\alpha_{21}+1}X_3^{\alpha_3-j}-X_2^{j\alpha_2+1}$ for $j=1,2,...,k$.
\begin{center}
\underline{For $k=l$:}
\end{center}
Now let $l$ be the smallest positive integer such that $l(\alpha_2+1)<(l-1)\alpha_1+(l+1)\alpha_{21}+\alpha_3$. Note that for any $j<l$, $j(\alpha_2+1)\geq (j-1)\alpha_1+(j+1)\alpha_{21}+\alpha_3$ and ${\rm LM}(f_{6+j})=X_1^{(j-1)\alpha_1+(j+1)\alpha_{21}+1}X_3^{\alpha_3-j}$ for $j=1,2,...,l-1$ and ${\rm LM}(f_{6+l})=X_2^{l\alpha_2+1}$ . Note also that $NF({\rm spoly}(f_m,f_n) \vert G)=0$ for $1\leq m <n\leq 6$ from the basis step since ${\rm LM}(f_{7})$ is not involved in calculations. Hence it is enough to prove $NF({\rm spoly}(f_m,f_n) \vert G)=0$ for all other $1\leq m <n\leq 6+ l$
\begin{itemize}
\item ${\rm spoly}(f_1,f_{6+j})=X_1^{j\alpha_1+(j+1)\alpha_{21}}X_3^{\alpha_3-j-1}-X_2^{j\alpha_2+1}X_4=h_1 $. ${\rm LM}(h_1)=X_2^{j\alpha_2+1}X_4$ by $(5)$ and $(6)$ and $T_{h_1}=\{f_5\}$.  ${\rm spoly}(h_1,f_{5})=X_1^{\alpha_{21}+1}X_2^{j\alpha_2}X_3^{\alpha_3-1}-X_1^{j\alpha_1+(j+1)\alpha_{21}+1}X_3^{\alpha_3-j-1}=h_2 $ and ${\rm LM}(h_2)=X_1^{\alpha_{21}+1}X_2^{j\alpha_2}X_3^{\alpha_3-1}$ by $(6)$. $T_{h_2}=\{f_6\}$ and  ${\rm spoly}(h_2,f_{6})=X_1^{\alpha_{1}+2\alpha_{21}}X_2^{(j-1)\alpha_{2}}X_3^{\alpha_3-2}-X_1^{j\alpha_{1}+(j+1)\alpha_{21}+1}X_3^{\alpha_3-j-1}= h_3$, ${\rm LM}(h_3)=X_1^{\alpha_{1}+2\alpha_{21}}X_2^{(j-1)\alpha_{2}}X_3^{\alpha_3-2}$ by $(6)$. $T_{h_3}=\{f_6\}$ and continuing inductively, $h_{j+1}={\rm spoly}(h_{j},f_{6})=X_1^{j\alpha_1+(j+1)\alpha_{21}}X_3^{\alpha_3-3}f_6$. Hence, $NF(\text{spoly}(f_1,f_{6+j}) \vert G)=0  $
\item ${\rm spoly}(f_2,f_{6+j})=X_1^{(j-1)\alpha_1+j\alpha_{21}+1}X_2^{\alpha_2}X_3^{\alpha_3-j}-X_2^{j\alpha_2+1}X_4=h_1 $.  ${\rm LM}(h_1)=X_2^{j\alpha_2+1}X_4$ by $(5)$ and $(6)$ and $T_{h_1}=\{f_5\}$. ${\rm spoly}(h_1,f_{5})=X_1^{\alpha_{21}+1}X_2^{j\alpha_{2}}X_3^{\alpha_3-1}-X_1^{(j-1)\alpha_1+j\alpha_{21}+1}X_2^{\alpha_{2}}X_3^{\alpha_3-j}=h_2.$ ${\rm LM}(h_2)=X_1^{\alpha_{21}+1}X_2^{j\alpha_{2}}X_3^{\alpha_3-1}$ by $(6)$. $T_{h_2}=\{f_6\}$ and  ${\rm spoly}(h_2,f_{6})$$=X_1^{\alpha_1+2\alpha_{21}+1}X_2^{(j-1)\alpha_2}X_3^{\alpha_3-2}-X_1^{(j-1)\alpha_1+j\alpha_{21}+1}X_2^{\alpha_2}X_3^{\alpha_3-j}=h_3$. ${\rm LM}(h_3)=X_1^{\alpha_1+2\alpha_{21}+1}X_2^{(j-1)\alpha_2}X_3^{\alpha_3-2}$ by $(6)$. $T_{h_3}=\{f_6\}$ and continuing inductively, we obtain $h_{j}={\rm spoly}(h_{j-1},f_{6})=X_1^{(j-2)\alpha_1+(j-1)\alpha_{21}+1}X_2^{\alpha_2}X_3^{\alpha_3-j}f_6$. Hence, $NF(\text{spoly}(f_2,f_{6+j}) \vert G)=0  $
\item ${\rm spoly}(f_3,f_{6+j})= X_1^{j\alpha_1+j\alpha_2}X_2-X_2^{j\alpha_2+1}X_3^{j}=h_1$.  ${\rm LM}(h_1)=X_2^{j\alpha_2+1}X_3^{j}$ by $(6)$ and $T_{h_1}=\{f_6\}$.  ${\rm spoly}(h_1,f_{6})=X_1^{\alpha_{1}+\alpha_{21}}X_2^{(j-1)\alpha_2+1}X_3^{j-1}-X_1^{j\alpha_1+j\alpha_{21}}X_2=h_2$. ${\rm LM}(h_2)=X_2^{(j-1)\alpha_2+1}X_3^{j-1}$ by $(6)$ and $T_{h_2}=\{f_6\}$. Continuing inductively ${\rm spoly}(h_{j-1},f_{6})=X_1^{(j-1)(\alpha_{1}+\alpha_{21})}X_2f_6=h_{j}$ and hence $NF(\text{spoly}(f_3,f_{6+j}) \vert G)=0  $

\item $NF(\text{spoly}(f_m,f_{6+j}) \vert G)=0  $ for $m=4,5$ as ${\rm LM}(f_m)$ and ${\rm LM}(f_{6+j})$ are relatively prime for all $1\leq j <l$

\item ${\rm spoly}(f_6,f_{6+j})=f_{6+(j+1)} $ and hence $NF(\text{spoly}(f_6,f_{6+j}) \vert G)=0  $ for all $1\leq j <l$
\item ${\rm spoly}(f_{6+i},f_{6+j})=X_1^{(j-i)\alpha_1+(j-i)\alpha_{21}}X_2^{i\alpha_2+1}-X_2^{j\alpha_{2}+1}X_3^{j-i}=h_1$. ${\rm LM}(h_1)=X_2^{j\alpha_{2}+1}X_3^{j-i}$ by $(6)$. $T_{h_1}=\{f_6\}$, ${\rm spoly}(h_{1},f_{6})=X_1^{\alpha_1+\alpha_{21}}X_2^{(j-1)\alpha_{2}+1}X_3^{j-i-1}-X_1^{(j-i)(\alpha_{1}+\alpha_{21})}X_2^{i\alpha_2+1}=h_2$ and ${\rm LM}(h_2)=X_1^{\alpha_1+\alpha_{21}}X_2^{(j-1)\alpha_{2}+1}X_3^{j-i-1}$ by $(6)$ and $T_{h_2}=\{f_6\}$. Continuing inductively $h_{j-i}={\rm spoly}(h_{j-i-1},f_{6})=X_1^{(j-i-1)(\alpha_1+\alpha_{21})}X_2^{i\alpha_2+1}f_6$ and hence $NF({\rm spoly}(f_{6+i},f_{6+j}) \vert G)=0$ for all  $1\leq i<j<l$.
\item $NF({\rm spoly}(f_m,f_{6+l}) \vert G)=0$ for $m=1,2,3,4,7,...,5+l$ as ${\rm LM}(f_m)$ and ${\rm LM}(f_{6+l})$ are relatively prime. 
\item ${\rm spoly}(f_5,f_{6+l})=  X_1^{\alpha_{21}+1}X_2^{l\alpha_2}X_3^{\alpha_3-1}-X_1^{(l-1)\alpha_1+(l+1)\alpha_{21}+1}X_3^{\alpha_3-l}X_4$. Set $h_1={\rm spoly}(f_5,f_{6+l})$. If ${\rm LM}(h_1)=X_1^{\alpha_{21}+1}X_2^{l\alpha_2}X_3^{\alpha_3-1}$. Then $T_{h_1}=\{f_6\}$ and ${\rm spoly}(h_1,f_6)=X_1^{\alpha_{1}+2\alpha_{21}+1}X_2^{(l-1)\alpha_2}X_3^{\alpha_3-2}-X_1^{(l-1)\alpha_1+(l+1)\alpha_{21}+1}X_3^{\alpha_3-l}X_4=h_2$. ${\rm LM}(h_2)=X_1^{(l-1)\alpha_1+(l+1)\alpha_{21}+1}X_3^{\alpha_3-l}X_4$ since $(l-1)(\alpha_2+1)\geq (l-2)\alpha_1+l\alpha_{21}+\alpha_3$ and $(5)$. Then $T_{h_2}=\{f_1,f_2\}$ and since ${\rm ecart}(f_2)$ is minimal by $(6)$,  ${\rm spoly}(h_2,f_2)=X_1^{\alpha_1+2\alpha_{21}+1}X_2^{(l-1)\alpha_2}X_3^{\alpha_3-2}-X_1^{(l-1)\alpha_1+l\alpha_{21}+1}X_2^{\alpha_2}X_3^{\alpha_3-l}=h_3$ and ${\rm LM}(h_3)=X_1^{\alpha_1+2\alpha_{21}+1}X_2^{(l-1)\alpha_2}X_3^{\alpha_3-2}$ by $(6)$.$T_{h_3}=\{f_6\}$ and  ${\rm spoly}(h_3,f_6)=X_1^{2\alpha_{1}+3\alpha_{21}+1}X_2^{(l-2)\alpha_2}X_3^{\alpha_3-3}-X_1^{(l-1)\alpha_1+l\alpha_{21}+1}X_2^{\alpha_2}X_3^{\alpha_3-l}=h_4$ and ${\rm LM}(h_4)=X_1^{2\alpha_{1}+3\alpha_{21}+1}X_2^{(l-2)\alpha_2}X_3^{\alpha_3-3}$ by $(6)$. $T_{h_4}=\{f_6\}$ and  ${\rm spoly}(h_4,f_6)=X_1^{3\alpha_{1}+4\alpha_{21}+1}X_2^{(l-3)\alpha_2}X_3^{\alpha_3-4}-X_1^{(l-1)\alpha_1+l\alpha_{21}+1}X_2^{\alpha_2}X_3^{\alpha_3-l}=h_5$ and ${\rm LM}(h_5)=X_1^{3\alpha_{1}+4\alpha_{21}+1}X_2^{(l-3)\alpha_2}X_3^{\alpha_3-4}$ by $(6)$ and $T_{h_5}=\{f_6\}$. Continuing inductively, $h_l={\rm spoly}(h_{l-1},f_6)=X_1^{(l-2)\alpha_{1}+(l-1)\alpha_{21}+1}X_2^{\alpha_2}X_3^{\alpha_3-l}f_6$ and hence $NF(\text{spoly}(f_5,f_{6+l}) \vert G)=0$\newline
Otherwise $T_{h_1}=\{f_1,f_2\}$ and ${\rm ecart}(f_2)$ is minimal by $(6)$. ${\rm spoly}(h_1,f_2)=X_1^{(l-1)\alpha_1+l\alpha_{21}+1}X_2^{\alpha_2}X_3^{\alpha_3-l}-X_1^{\alpha_{21}+1}X_2^{l\alpha_2}X_3^{\alpha_3-1}=h_2 $ and ${\rm LM}(h_2)=X_1^{\alpha_{21}+1}X_2^{l\alpha_2}X_3^{\alpha_3-1}$ by $(6)$. Since $T_{h_2}=\{f_6\}$, we compute  ${\rm spoly}(h_2,f_6)=X_1^{(l-1)\alpha_1+l\alpha_{21}+1}X_2^{\alpha_2}X_3^{\alpha_3-l}-X_1^{\alpha_1+2\alpha_{21}+1}X_2^{(l-1)\alpha_2}X_3^{\alpha_3-2}=h_3$, $T_{h_3}=\{f_6\}$. Continuing inductively, $h_l=$ ${\rm spoly}(h_{l-1},f_6)=X_1^{(l-1)\alpha_1+l\alpha_{21}+1}X_2^{\alpha_2}X_3^{\alpha_3-l}-X_1^{(l-2)\alpha_1+(l-1)\alpha_{21}+1}X_2^{2\alpha_2}X_3^{\alpha_3-l+1}$$=X_1^{(l-2)\alpha_1+(l-1)\alpha_{21}+1}X_2^{\alpha_2}X_3^{\alpha_3-l}f_6$ and hence $NF(\text{spoly}(f_5,f_{6+l}) \vert G)=0$
\item ${\rm spoly}(f_6,f_{6+l})=X_1^{\alpha_1+\alpha_{21}}f_{6+(l-1)}  $ hence, $NF(\text{spoly}(f_6,f_{6+l}) \vert G)=0$. 
\end{itemize}
Since all normal forms reduce to zero, $\{{f_1},{f_2}, f_3, f_4, f_5,{f_6},...,{f_{6+k}}\}$ is a standard basis  for $I_C$ 
\end{proof}
\begin{corollary}
$\{{f_1}^*,{f_2}^*,...,{f_6}^*,...,{f_{6+k}}^*\}$ is a standard basis  for $I_C^*$  where  ${f_1}^*=X_3X_4$, ${f_2}^*=X_1^{\alpha_{21}}X_4$,  ${f_3}^*=X_3^{\alpha_3}$, ${f_4}^*=X_4^{2}$, ${f_5}^*=X_2X_4$, ${f_6}^*=X_2^{\alpha_2}X_3$, $f_{6+k}^*=X_2^{k\alpha_2+1}$ and ${f_{6+j}}^*=X_1^{(j-1)\alpha_1+(j+1)\alpha_{21}+1}X_3^{\alpha_3-j}$ for $j=1,2,...,k-1$. Since $X_1 | {f_2}^*$, the tangent cone is not Cohen-Macaulay by the criterion given in \cite{AMS} as expected.
\end{corollary}



\section{Hilbert Function}\label{3}
\begin{theorem}
The numerator of the Hilbert series of the local ring $R_S$ is 
$$P(I_S^*)=1-3t^2+3t^3-t^4-t^{\alpha_{21}+1}(1-t)^3-t^{\alpha_3}(1-t)-t^{\alpha_2+1}(1-t)(1-t^{\alpha_3-1})-t^{k\alpha_2+1}(1-t)^2-r(t)$$ where $r(t)=0$ if $k=1$ and $r(t)=\sum_{j=2}^{k} t^{(k-j)\alpha_1+(k-j+2)\alpha_{21}+\alpha_3+j-k}(1-t)^2(1-t^{\alpha_2})$ otherwise.
\end{theorem}
\begin{proof}
To compute the Hilbert series, we will use Algorithm 2.6 of \cite{bayer} that is formed by continuous use of the proposition 

"If $I$ is a monomial ideal with $I=<J,w>$, then the numerator of the Hilbert series of $A/I$ is $P(I)=h(J)-t^{\deg w}h(J:w)$ where $w$ is a monomial and $\deg w$ is the total degree of $w$."

Let $P(I_S^*)$ denote the numerator of the Hilbert series of $A/{I^*_S}$
\begin{itemize}
\item Let $w_1=X_2^{k\alpha_2+1}$, then $$J_1=<X_3X_4,X_1^{\alpha_{21}}X_4,X_3^{\alpha_3}, X_4^{2}, X_2X_4, X_2^{\alpha_2}X_3, X_1^{2\alpha_{21}+1}X_3^{\alpha_3-1},...,X_1^{(k-2)\alpha_1+(k)\alpha_{21}+1}X_3^{\alpha_3-k+1} >,$$
$P(I_S^*)=P(J_1)-t^{k\alpha_2+1}P(<X_4,X_3>)=P(J_1)-t^{k\alpha_2+1}(1-t)^2$
\end{itemize}
If $k=1$ then $w_2=w_{k+1}$, otherwise:
\begin{itemize}
\item Let $w_2=X_1^{(k-2)\alpha_1+(k)\alpha_{21}+1}X_3^{\alpha_3-k+1}$, then $$J_2=<X_3X_4,X_1^{\alpha_{21}}X_4,X_3^{\alpha_3}, X_4^{2}, X_2X_4, X_2^{\alpha_2}X_3, X_1^{2\alpha_{21}+1}X_3^{\alpha_3-1},...,X_1^{(k-3)\alpha_1+(k-1)\alpha_{21}+1}X_3^{\alpha_3-k+2} >,$$
\begin{eqnarray*}
P(J_1)&=&P(J_2)-t^{(k-2)\alpha_{1}+k(\alpha_{21}-1)+\alpha_3+2}P(<X_4,X_2^{\alpha_2},X_3>)\\
&=&P(J_2)-t^{(k-2)\alpha_{1}+k(\alpha_{21}-1)+\alpha_3+2}(1-t)^2(1-t^{\alpha_2})
\end{eqnarray*}
\item Continue inductively and let $w_k=X_1^{2\alpha_{21}+1}X_3^{\alpha_3-1}$, then $$J_{k}=<X_3X_4,X_1^{\alpha_{21}}X_4,X_3^{\alpha_3}, X_4^{2}, X_2X_4, X_2^{\alpha_2}X_3 >,$$
$P(J_{k-1})=P(J_k)-t^{2\alpha_{21}+\alpha_{3}}P(<X_4,X_2^{\alpha_2},X_3>)=P(J_k)-t^{2\alpha_{21}+\alpha_{3}}(1-t)^2(1-t^{\alpha_2})$
\item Let $w_{k+1}=X_2^{\alpha_2}X_3$, then 
$$J_{k+1}=<X_3X_4,X_1^{\alpha_{21}}X_4,X_3^{\alpha_3}, X_4^{2}, X_2X_4 >,$$
$P(J_{k})=P(J_{k+1})-t^{\alpha_2+1}P(<X_4,X_3^{\alpha_3-1}>)=P(J_{k+1})-t^{\alpha_2+1}(1-t)(1-t^{\alpha_3-1})$
\item Let $w_{k+2}=X_3^{\alpha_3}$, then 
$$J_{k+2}=<X_3X_4,X_1^{\alpha_{21}}X_4, X_4^{2}, X_2X_4 >,$$
$P(J_{k+1})=P(J_{k+2})-t^{\alpha_3}P(<X_4>)=P(J_{k+2})-t^{\alpha_3}(1-t)$
\item Let $w_{k+3}=X_1^{\alpha_{21}}X_4$, then 
$$J_{k+3}=<X_3X_4, X_4^{2}, X_2X_4>$$
$P(J_{k+2})=P(J_{k+3})-t^{\alpha_{21}+1}h(<X_3,X_4,X_2>)=P(J_{k+3})-t^{\alpha_{21}+1}(1-t)^3$
\item Let $w_{k+4}=X_2X_4$, then 
$$J_{k+4}=<X_3X_4, X_4^{2}>$$
$P(J_{k+3})=P(J_{k+4})-t^{2}P(<X_3,X_4>)=P(J_{k+4})-t^{2}(1-t)^2$
\item Let $w_{k+5}=X_4^{2}$, then 
$$J_{k+5}=<X_3X_4>$$
$P(J_{k+4})=P(J_{k+5})-t^{2}P(<X_3>)=(1-t^2)-t^2(1-t)$
\end{itemize} 
Hence, $P(I_S^*)=(1-t^2)-t^2(1-t)-t^{2}(1-t)^2-t^{\alpha_{21}+1}(1-t)^3-t^{\alpha_3}(1-t)-t^{\alpha_2+1}(1-t)(1-t^{\alpha_3-1})-t^{k\alpha_2+1}(1-t)^2-r(t) $ where $r(t)=0$ if $k=1$ and $r(t)=\sum_{j=2}^{k} t^{(k-j)\alpha_1+(k-j+2)\alpha_{21}+\alpha_3+j-k}(1-t)^2(1-t^{\alpha_2})$ if $k>1$
\end{proof}
Clearly, since the krull dimension is one, if there are no negative terms in second Hilbert Series, then the Hilbert function will be non-decreasing. We can state and prove the next theorem.
\begin{theorem}
The local ring $R_S$ has a non-decreasing Hilbert function if $k=1$.
\end{theorem}
\begin{proof}
Observe that $$P(I_S^*)=(1-t)P_1(t)$$ with $P_1(t)=1+t-t^2-t^{2}(1-t)-t^{\alpha_{21}+1}(1-t)^2-t^{\alpha_3}-t^{\alpha_2+1}(1-t^{\alpha_3-1})-t^{k\alpha_2+1}(1-t)-r_1(t) $ where $r_1(t)=0$ if $k=1$ and $r_1(t)=\sum_{j=2}^{k} t^{(k-j)\alpha_1+(k-j+2)\alpha_{21}+\alpha_3+j-k}(1-t)(1-t^{\alpha_2})$ if $k>1$.
$$P_1(t)=(1-t)P_2(t)$$ with $P_2(t)=1+t+t(1+t+...+t^{\alpha_3-2})-t^2-t^{\alpha_{21}+1}(1-t)-t^{\alpha_2+1}(1+t+...+t^{\alpha_3-2})-t^{k\alpha_2+1}-r_2(t)$ where $r_2(t)=0$ if $k=1$ and $r_2(t)=\sum_{j=2}^{k} t^{(k-j)\alpha_1+(k-j+2)\alpha_{21}+\alpha_3+j-k}(1-t^{\alpha_2})$ if $k>1$. Combining some terms,
$P_2(t)=1-t^2+t(1-t^{k\alpha_2})+t(1-t^{\alpha_2})(1+t+...+t^{\alpha_3-2})-t^{\alpha_{21}+1}(1-t)-r_2(t)$ which shows that,
$$P_2(t)=(1-t)Q(t)$$
with 
$$Q(t)=1+t+t(1+t+...+t^{k\alpha_2-1})+t(1+t+...+t^{\alpha_2-1})(1+t+...+t^{\alpha_3-2})-t^{\alpha_{21}+1}+r_3(t)$$
where $r_3(t)=0$ if $k=1$ and $r_3(t)=\sum_{j=2}^{k} t^{(k-j)\alpha_1+(k-j+2)\alpha_{21}+\alpha_3+j-k}(1+t+...+t^{\alpha_2-1})$ if $k>1$.Then 
$HS_{R}(t)=\frac{P(t)}{(1-t)^4}=\frac{Q(t)}{(1-t)}$, and $Q(t)=1+t+t(1+t+...+t^{k\alpha_2-1})+t(1+t+...+t^{\alpha_2-1})(1+t+...+t^{\alpha_3-2})-t^{\alpha_{21}+1}+r_3(t)$ is the second Hilbert series of the local ring. Note that $\alpha_2>\alpha_{21}+1$ and $-t^{\alpha_{21}+1}$ disappear when we expand $Q(t)$. Hence for $k=1$, there are no negative coefficients in the second Hilbert series, the Hilbert series in non-decreasing.
\end{proof}
\begin{remark}
 For $k>1,$ $$Q(t)=1+t-t^{\alpha_{21}+1}+(1+t+...+t^{\alpha_2-1})\left[t(2+t+...+t^{\alpha_3-2})+\sum_{j=1}^{k-1}S_j(t) \right]$$
 where $S_j(t)=t^{(j)\alpha_{2}+1}-t^{(j-1)\alpha_1+(j+1)\alpha_{21}+\alpha_3+1-j}$. Recall that  that $k$ is the smallest positive integer such that $k(\alpha_2+1)<(k-1)\alpha_1+(k+1)\alpha_{21}+\alpha_3$ hence for any $1\leq j \leq k-1$, we have $j(\alpha_2+1)\geq (j-1)\alpha_1+(j+1)\alpha_{21}+\alpha_3$, which means $S_j(t)\geq 0$ for ant $t\in \mathbb{N}$.
 
\end{remark}

\section{Examples}\label{4}
Above examples are done via the computer algebra system SINGULAR, \cite{singular}.
\begin{example}For $\alpha_{21}=8 , \alpha_1=16 , \alpha_2=20 ,\alpha_3=7, \alpha_{4}=2 $, $k=1$ and the corresponding standard basis is $\{ f_1=X_1^{16}-X_3X_4,
f_2=X_2^{20}-X_1^8X_4,
f_3=X_3^7-X_1^7X_2,
f_4=X_4^2-X_1X_2^{19}X_3^6,
f_5=X_1^9X_3^6-X_2X_4,
f_6=X_1^{24}-X_2^{20}X_3,
f_7=X_1^{17}X_3^6-X_2^{21} \}$. Numerator of the Hilbert series of the tangent cone is $P(I_S^*)=1-3t^2+3t^3-t^4-t^7+t^8-t^9+3t^{10}-3t^{11}+t^{12}-2t^{21}+3t^{22}-t^{23}+t^{27}-t^{28}$. Direct computation shows that the Hilbert series is non-decreasing.
\end{example}
\begin{example}
For $\alpha_{21}=4$, $\alpha_{1}=22$,$\alpha_{2}=13$, $\alpha_{3}=5$, $\alpha_{4}=2$, we have $k=2$ and the corresponding standard basis is $\{ f_1=X_1^{22}-X_3X_4
,f_2=X_2^{13}-X_1^4X_4
,f_3=X_3^5-X_1^{17}X_2
,f_4=X_4^2-X_1X_2^{12}X_3^4
,f_5=X_1^5X_3^4-X_2X_4
,f_6=X_1^{26}-X_2^{13}X_3
,f_7=X_1^9X_3^4-X_2^{14}
,f_8=X_1^{35}X_3^3-X_2^{27}  \}$.  Numerator of the Hilbert series of the tangent cone is $P(I_S^*)=1-3t^2+3t^3-t^4-2t^5+4t^6-3t^7+t^8-t^{13}+t^{14}+t^{18}-t^{19}+t^{26}-3t^{27}+3t^{28}-t^{29}.$ Direct computation shows that the Hilbert series is non-decreasing.

\end{example}
\begin{example}
For $\alpha_{21}=10$, $\alpha_{1}=17$,$\alpha_{2}=25$, $\alpha_{3}=4$, $\alpha_{4}=2$, we have $k=3$ and the corresponding standard basis is $\{f_1= X_1^{17}-X_3X_4,f_2= X_2^{25}-X_1^{10}X_4,f_3= X_3^4-X_1^6X_2, f_4=X_4^2-X_1X_2^{24}X_3^3,f_5= X_1^{11}X_3^3-X_2X_4, f_6=X_1^{27}-X_2^{25}X_3, f_7=X_2^{26}-X_1^{21}X_3^3, f_8=X_1^{48}X_3^2-X_2^{51},f_9=X_1^{75}X_3-X_2^{76}\}$.  Numerator of the Hilbert series of the tangent cone is $P(I_S^*)=1-3t^2+3t^3-2t^4+t^5-t^{11}+3t^{12}-3t^{13}+t^{14}-t^{24}+2t^{25}-2t^{26}+t^{27}+t^{29}-t^{30}+t^{49}-3t^{50}+3t^{51}-t^{52}+t^{75}-3t^{76}+3t^{77}-t^{78}$. Direct computation shows that the Hilbert series is non-decreasing.
\end{example}

\begin{example}For $\alpha_{21}=3$, $\alpha_{1}=13$,$\alpha_{2}=14$, $\alpha_{3}=6$, $\alpha_{4}=2$, we have $k=4$ and the corresponding standard basis becomes $\{ f_1=X_1^{13}-X_3X_4,
f_2=X_2^{14}-X_1^3X_4,
f_3=X_3^6-X_1^9X_2,
f_4=X_4^2-X_1X_2^{13}X_3^5,
f_5=X_1^4X_3^5-X_2X_4,
f_6=X_1^{16}-X_2^{14}X_3,
f_7=X_1^7X_3^5-X_2^{15},
f_8=X_1^{23}X_3^4-X_2^{29},
f_9=X_1^{39}X_3^3-X_2^{43},
f_{10}=X_1^{55}X_3^2-X_2^{57} \}$.  Numerator of the Hilbert series of the tangent cone is $P(I_S^*)=1-3t^2+3t^3-2t^4+3 t^{5}-4t^{6}+2t^{7}-t^{12}+2t^{13}-t^{14}-t^{15}+t^{16}+t^{20}-t^{21}+t^{26}-3t^{27}+3t^{28}-t^{29}+t^{41}-3t^{42}+3t^{43}-t^{44}+t^{56}-3t^{57}+3t^{58}-t^{59}$. Direct computation shows that the Hilbert series is non-decreasing.
\end{example}




\end{document}